\documentclass[12pt]{amsart}
\usepackage{amssymb}
\usepackage{amsmath}
\usepackage{amsthm}
\usepackage{amsbsy}
\usepackage{psfrag}
\usepackage{pstricks}
\usepackage{graphics}
\usepackage{graphicx}
\theoremstyle{plain}
\newtheorem{theorem}{Theorem}[section]
\newtheorem{lemma}[theorem]{Lemma}

\theoremstyle{remark}
\newtheorem{remark}[theorem]{Remark}
\newtheorem{algorithm}[theorem]{Algorithm}

\begin{document}
\allowdisplaybreaks[4]
\numberwithin{figure}{section}
\numberwithin{table}{section}
 \numberwithin{equation}{section}
%
\title[A Quadratic FEM for the 3D-Obstacle Problem]
 {Bubbles Enriched Quadratic Finite Element Method for the 3D-Elliptic Obstacle Problem}

\author{Sharat Gaddam}
\address{Department of Mathematics, Indian Institute of Science, Bangalore - 560012}
\email{sharat12@math.iisc.ernet.in}

\author{Thirupathi Gudi}
\address{Department of Mathematics, Indian Institute of Science, Bangalore - 560012}
\email{gudi@math.iisc.ernet.in}

\date{}
\begin{abstract}
Optimally convergent (with respect to the regularity) quadratic
finite element method for two dimensional obstacle problem on
simplicial meshes is studied in (Brezzi, Hager, Raviart, Numer.
Math, 28:431--443, 1977). There was no analogue of a quadratic
finite element method on tetrahedron meshes for three dimensional
obstacle problem. In this article, a quadratic finite element
enriched with element-wise bubble functions is proposed for the
three dimensional elliptic obstacle problem. A priori error
estimates are derived to show the optimal convergence of the
method with respect to the regularity. Further a posteriori error
estimates are derived to design an adaptive mesh refinement
algorithm. Numerical experiment illustrating the theoretical
result on {\em a priori} error estimate is presented.
\end{abstract}
\maketitle
\allowdisplaybreaks
\def\R{\mathbb{R}}
\def\cA{\mathcal{A}}
\def\cK{\mathcal{K}}
\def\cN{\mathcal{N}}
\def\p{\partial}
\def\O{\Omega}
\def\bbP{\mathbb{P}}
\def\cV{\mathcal{V}}
\def\cM{\mathcal{M}}
\def\cT{\mathcal{T}}
\def\cE{\mathcal{E}}
\def\bF{\mathbb{F}}
\def\bC{\mathbb{C}}
\def\cR{\mathcal{R}}
\def\bN{\mathbb{N}}
\def\ssT{{\scriptscriptstyle T}}
\def\HT{{H^2(\O,\cT_h)}}
\def\mean#1{\left\{\hskip -5pt\left\{#1\right\}\hskip -5pt\right\}}
\def\jump#1{\left[\hskip -3.5pt\left[#1\right]\hskip -3.5pt\right]}
\def\smean#1{\{\hskip -3pt\{#1\}\hskip -3pt\}}
\def\sjump#1{[\hskip -1.5pt[#1]\hskip -1.5pt]}
\def\jumptwo{\jump{\frac{\p^2 u_h}{\p n^2}}}

\section{Introduction}\label{sec:Intro}
The obstacle problem appears in the study of elliptic variational
inequalities with applications in contact mechanics, option
pricing and fluid flow problems. Generally, the obstacle problem
exhibits free boundary along which the regularity of the solution
is influenced. The location of the free boundary is not {\em a priori}
known and it forms a part of the numerical approximation. This
makes the finite element approximation of this problem an
interesting subject as it offers challenges both in the theory and
the computation. we refer to the books
\cite{AH:2009:VI,Glowinski:2008:VI,KS:2000:VI,Sutt:book} for the
theoretical and numerical aspects of variational inequalities. The
finite element analysis of the obstacle problem started in 1970's,
see \cite{BHR:1977:VI,Falk:1974:VI}. Subsequently there has been a
tremendous progress on the subject, see
\cite{BLY:2012:C0IP,BLY:2012:C0IP1,HK:1994:multiadaptive,Wang:2002:P2VI,WHC:2010:DGVI}
for the convergence analysis of finite element methods for the
obstacle problem and see
\cite{Belgacem:2000:SINUM,DH:2015:Signorini,HR:2012:Signorini,HW:2005:Signorini}
for the Signorini contact problem. The adaptive finite element
methods play an important role in improving the accuracy of the
numerical solution in an efficient way. A posteriori error
estimates are key tools in the design of adaptive schemes, see
\cite{AO:2000:Book} for the theory of a posteriori error analysis.
In the context of the obstacle problem there has been a lot of
work, see
\cite{AOL:1993:Apost,BC:2004:VI,Braess:2005:VI,CN:2000:VI,GKVZ:2011:VeeserHirarchy,TG:2015:VIP2,NSV:2005:Apost,NPZ:2010:VI,Veeser:2001:VI,WW:2010:Apost}
and see
\cite{BS:2014:hp-apost,TG:2014:VIDG,TG:2014:VIDG1,Gwinner:2009:pfem,Gwinner:2013:pfem,WHE:2015:ApostDG}{}.
Further,  the convergence of adaptive methods based on a
posteriori error estimates is also studied recently, see
\cite{BCH:2007:VI,BCH:2009:AVI,FPP:2014:Obstacle,SV:2007:APost,PP:2013:AconvVI}.
Further, we refer to
\cite{BS:2000:VI,HN:2005:Signorini,AS:2011:Contact,WW:2009:Contact}
for the work related to the Signorini contact problem.

The contribution of this article is on the design and analysis of
a quadratic finite element method for the three dimensional
elliptic obstacle problem. The work in
\cite{BHR:1977:VI,Wang:2002:P2VI} and \cite{TG:2015:VIP2} is for a
quadratic finite element method (FEM) for the two dimensional
obstacle problem. The quadratic FEM in two dimensions is based on
the discrete constraints at the midpoints of the edges of the
triangles. These constraints are shown to be enough to guarantee
the convergence of the method at the rate that is optimal with
respect to the regularity of the solution. The key idea in a
priori error estimates in \cite{BHR:1977:VI,Wang:2002:P2VI} can
realized to be is that if a quadratic function $v$ is nonnegative
at the midpoints of a triangle $T$, then the integral of $v$ on
$T$ is nonnegative. This is a simple fact from the observation
that the integral of a canonical $P_2$-nodal basis function
correspond to a vertex on $T$ is zero. This guides to consider the
constraints at the midpoints of the edges only. However the same
principle cannot be extended to three dimensional domains as the
integral of a canonical $P_2$-nodal basis function corresponding
to a vertex is negative.
The remedy we adopt in this article is by enriching the
$P_2$-finite element space with element-wise bubble functions and
then considering the constraints on the integral mean values over
each simplex in the mesh. The a priori error analysis is performed
to show the convergence of the scheme. Further a posteriori error
estimates are derived to design an adaptive finite element scheme.
In the literature, there are $hp$-finite element methods
available for the obstacle problem
\cite{BS:2014:hp-apost,Gwinner:2009:pfem,Gwinner:2013:pfem}, but
they use rectangular elements which are not well-suited for the
adaptive mesh refinement algorithms.

Let $\Omega\subset \R^3$ be a bounded polyhedral domain with
boundary $\partial\Omega$. Assume that the load function $f\in
L^2(\Omega)$ and the obstacle $\chi \in C(\bar\Omega)\cap
H^1(\Omega)$ satisfying $\chi|_{\partial\Omega}\leq 0$. We will
also assume additional regularity on $f$ and $\chi$ in the subsequent a
priori error analysis. The admissible closed and convex set for
the solution is defined by
\begin{align*}
\cK=\{v\in H^1_0(\Omega): v\geq \chi \text{ a.e. in } \Omega\}.
\end{align*}
Note that since $\chi^+=\max\{\chi,0\}\in \cK$, the set $\cK$ is
nonempty. We consider the model problem of finding $u\in \cK$ such
that
\begin{align}\label{eq:MP}
a(u,v-u)\geq (f,v-u)\;\; \text{ for all } v\in \cK,
\end{align}
where for simplicity $a(u,v)=(\nabla u,\nabla v)$. Hereafter,
$(\cdot,\cdot)$ denotes the $L^2(\Omega)$ inner-product. We denote
by  $\|\cdot\|$  the $L^2(\Omega)$ norm. The  result of
Stampacchia \cite{AH:2009:VI,Glowinski:2008:VI,KS:2000:VI} implies
the existence of a unique solution to \eqref{eq:MP}.

\par
\noindent For the  a posteriori error analysis, we make use of the
Lagrange multiplier $\sigma\in H^{-1}(\Omega)$ defined by
\begin{align}\label{eq:sigmadef}
\langle \sigma, v\rangle =(f,v)-a(u,v)\;\;\text{ for all } v\in
H^1_0(\Omega),
\end{align}
where $\langle \cdot,\cdot\rangle$ denotes the duality bracket of
$H^{-1}(\Omega)$ and $H^1_0(\Omega)$. It is useful to note from
\eqref{eq:sigmadef} and \eqref{eq:MP} that
\begin{equation}\label{eq:sigma}
\langle \sigma, v-u\rangle\leq 0\; \text{ for all } v \in \cK.
\end{equation}

The rest of the article is organized as follows. In the Section
\ref{sec:DP}, we introduce the notation, preliminaries, the
discrete problem and the Lagrange multiplier for a posteriori
error estimates. In the Section \ref{sec:Apriori} and
\ref{sec:Aposteriori}, we derive a priori and a posteriori error
estimates, respectively. In the Section
\ref{sec:Numerical}, we propose a primal dual active set algorithm
for solving the discrete problem and subsequently present a
numerical experiment. Finally we conclude the article in the Section
\ref{sec:Conclusions}.

\section{Discrete Problem}\label{sec:DP}
\subsection{Preliminaries.} Let $\cT_h$ be a regular triangulation of $\Omega$ with simplices (tetrahedrons).
A generic tetrahedron (simplex) is denoted by $T$ and its diameter
and volume by $h_T$ and $|T|$, respectively. Set $h=\max\{h_T :
T\in\cT_h\}$. The set of all vertices of tetrahedrons that are
inside $\Omega$ is denoted by $\cV_h^i$. The set of all vertices
that are on the boundary $\partial\Omega$ is denoted by $\cV_h^b$.
Set $\cV_h=\cV_h^i\cup\cV_h^b$. We also use $\cV_T$ to denote the
set of four vertices of the tetrahedron $T$. Let $\cM_h^i$ (resp.
$\cM_h^b$) be the set of all midpoints of the interior (resp. boundary) edges of
$\cT_h$ and set $\cM_h=\cM_h^i\cup\cM_h^b$. Further, we denote the
set of midpoints of the six edges of $T$ by $\cM_T$. The set of
all interior faces is denoted by $\cE_h^i$. Finally, we denote the
diameter of a generic face $e\in\cE_h^i$ by $h_e$.
\par
For any $e\in\cE_h^i$, there are two simplices $T_+$ and $T_-$
such that $e=\partial T_+\cap\partial T_-$. Let $n_-$ be the unit
normal of $e$ pointing from $T_-$ to $T_+$, and $n_+=-n_-$. For
any $v$ which is piecewise smooth, we define the jump of $\nabla
v$ on $e$ by
\begin{eqnarray*}
\sjump{\nabla v} = \nabla v_-\cdot n_-+\nabla v_+\cdot n_+,
\end{eqnarray*}
where $v_\pm=v|_{T_\pm}$ and $v|_T$ denotes the restriction of $v$
to $T$.

\par
For any $T\in\cT_h$ and $v\in L^1(T)$, define
\begin{align*}
A_T(v)= \frac{1}{|T|}\int_T v(x)\,dx.
\end{align*}

\par
Let $V_{pc,h}=\{v\in L^1(\Omega): v|_T\in \bbP_0(T)\text{ for all
} T\in\cT_h\},$ where $\bbP_r(T)$ denotes the space of polynomials
of total degree less than or equal to $r$. Define
$A_h:L^1(\Omega)\rightarrow V_{pc,h}$ by $A_h(v)|_T=A_T(v)$ for
all $v\in L^1(\Omega)$.

\subsection{Discrete Problem}

Before defining the finite element space, we define for each
simplex $T\in \cT_h$ a $P_4(T)$ bubble function $b_T$ by
\begin{align}\label{eq:bT}
b_T=256\,\lambda_1^T\lambda_2^T\lambda_3^T\lambda_4^T,
\end{align}
where $\lambda_i^T$( for $i=1,2,3,4$) is the barycentric
coordinate of $T$ associated with the vertex $a_i \in\cV_T$.
Define the spaces
$$W_h=\{v_h\in H^1_0(\Omega): v_h|_{T} \in \bbP_2(T)\text{ for
all } T\in\cT_h\},$$
 and
$$B_h=\{v_h\in H^1_0(\Omega): v_h|_{T}\in \text{span}\{b_T\}  \text{ for
all } T\in\cT_h\}.$$
\smallskip
\par\noindent The finite element space  $V_h$ for approximating the
obstacle problem is defined by
$$V_h=W_h\oplus B_h.$$

\par\noindent
Define the discrete set
\begin{align*}
\cK_h=\{v_h\in V_h: A_h(v_h) \geq A_h(\chi)\}.
\end{align*}
The discrete problem consists of finding $u_h\in\cK_h$ such that
\begin{align}\label{eq:DP}
a(u_h,v_h-u_h)\geq (f,v_h-u_h)\,\text{ for all } v_h\in \cK_h.
\end{align}

\par
\noindent In the subsequent discussion we show that the above
discrete problem \eqref{eq:DP} has a unique solution by showing
that the discrete set $\cK_h$ is non-empty.

\par
\noindent {\bf Interpolation} $I_h$: Define an interpolation
operator $I_h:C(\bar\Omega)\rightarrow V_h$ by the following: Let
$v\in C(\bar\Omega)$ and define $I_hv$ by its nodal values
\begin{align}
I_hv(p)&=v(p) \quad\forall\; p\in \cV_h\cup\cM_h, \\
A_T(I_hv)&=A_T(v) \quad\;\forall T\in\cT_h.
\end{align}
\par
\noindent Define $I_T$ by $I_Tv=(I_hv)|_T$ for $v\in
C(\bar\Omega)$. The interpolation operator $I_h$ is well-defined
and satisfies $I_Tv=v$ for any $v\in \bbP_2(T)$. Therefore the
following approximation properties hold by the Bramble-Hilbert
Lemma and scaling \cite{BScott:2008:FEM,Ciarlet:1978:FEM}:
\begin{lemma}\label{lem:Approx}
Let $v\in H^s(T)$ for $2\leq s\leq 3$ and $T\in\cT_h$ . Then
\begin{align*}
|v-I_Tv|_{H^m(T)}&\leq C h_T^{s-m} |v|_{H^s(T)}, for \quad 0\leq
m\leq s,\\
\|v-A_T(v)\|_{L^2(T)}&\leq C h_T^r |v|_{H^r(T)},
\end{align*}
where $0\leq r\leq 1$.
\end{lemma}

\par
\noindent We remark here that in the subsequent {\em a priori}
error analysis, the interpolation $I_h$ gives good control for the
terms near the free boundary, see for example \eqref{eq:FreeBterm}, apart from
preserving the integral sign.

\par Since $u\geq \chi$, it is clear that $I_hu\in \cK_h$
and hence the set $\cK_h$ is nonempty. Now as in the case of
continuous problem \eqref{eq:MP}, the discrete problem
\eqref{eq:DP} can be shown to have a unique solution.
The {\em a posteriori} error analysis will make use of a discrete
Lagrange multiplier $\sigma_h$ analogous to $\sigma$ in
\eqref{eq:sigmadef}. Before defining it, we note the following
facts about the discrete solution $u_h$:

\par
Let $z_h\in V_h$ with $A_h(z_h)\geq 0$. Then, we have $u_h+z_h \in
\cK_h$. By taking $v_h=u_h+z_h$ in \eqref{eq:DP}, we find
\begin{align}\label{eq:DPProperty}
a(u_h,z_h)\geq (f,z_h).
\end{align}
Let $v_h\in V_h$ with $A_h(v_h)=0$. Then, by taking $z_h =\pm v_h$
in \eqref{eq:DPProperty}, we find
\begin{align}\label{eq:DPProperty1}
a(u_h,v_h)=(f,v_h).
\end{align}
Suppose for any $T\in\cT_h$,  $A_T(u_h)>A_T(\chi)$. Then by taking
$v_h^\pm=u_h\pm \delta b_T$ for some sufficiently small
$\delta>0$, we find
\begin{align*}
a(u_h,b_T)=(f,b_T),
\end{align*}
where $b_T$ is the bubble function defined in \eqref{eq:bT} on $T$
and extended by zero on $\bar\Omega\backslash T$. Therefore
\begin{align}\label{eq:DPProperty2}
a(u_h,b_T)=(f,b_T) \;\; \text{ for all }\; T \in \{T' \in\cT_h
:A_{T'} (u_h)>A_{T'}(\chi)\}.
\end{align}

\begin{lemma}\label{lem:Pih}
The map $\Pi_h:V_h\rightarrow V_{pc,h}$ defined by
$\Pi_h(v_h)=A_h(v_h)$  is onto and hence an inverse map
$\Pi_h^{-1}:V_{pc,h}\rightarrow V_h$ can be defined into a subset
of $V_h$ such that $\Pi^{-1}_h(w_h)=v_h$ where $v_h\in V_h$ with
$A_h(v_h)=w_h$ for $w_h\in V_{pc,h}$.
\end{lemma}

\begin{proof}
For given any $w_h\in V_{pc,h}$, we prove that there is some
$v_h\in V_h$ such that $A_h(v_h)=w_h$. Note that, we can write
$v_h\in V_h$ as $v_h=v_1+v_2$, where $v_1\in W_h$ and $v_2\in
B_h$. We choose first some $v_1\in W_h$, and then we choose
$v_2\in B_h$ such that $A_h(v_2)=w_h-A_h(v_1)$.  In particular, we
can choose $v_1$ to be zero and $v_2$ to be such that $v_2\in B_h$
with $v_2|_T=w_h b_T/A_h(b_T)$. This proves that $\Pi_h$ is onto.
Define $\Pi_h^{-1}$ by $\Pi_h^{-1}(w_h)=v_h$ where $v_h\in V_h$ is
such that $A_h(v_h)=w_h$ for $w_h\in V_{pc,h}$. Again $v_h$ can be
chosen such that $v_h|_T=w_h b_T/A_h(b_T)$.
\end{proof}

\par
\noindent Define the discrete Lagrange multiplier $\sigma_h\in
V_{pc,h}$ by
\begin{align}\label{eq:sigmahdef}
(\sigma_h,w_h)=(f,\Pi_h^{-1}w_h)-a(u_h,\Pi_h^{-1}w_h)\quad
\forall\, w_h\in V_{pc,h},
\end{align}
where $\Pi_h^{-1}$ is defined as in Lemma \ref{lem:Pih}.

\par
\noindent
The following lemma proves some properties of $\sigma_h$.

\begin{lemma}\label{lem:sigmaprops}
The discrete Lagrange multiplier defined by $\eqref{eq:sigmahdef}$
is well-defined. Further
\begin{align}\label{eq:sigmahP1}
\sigma_h\leq 0  \text{  on  } \bar\Omega,
\end{align}
and
\begin{align}\label{eq:sigmhaP2}
\sigma_h|_T=0\;\; \text{ for all }\; T \in \{T' \in\cT_h :A_{T'}
(u_h)>A_{T'}(\chi)\}.
\end{align}
\end{lemma}

\begin{proof}
For $w_h\in V_{pc,h}$, let $v_1$ and $v_2\in V_h$ be such that
$v_1\neq v_2$ and $\Pi_h(v_1)=\Pi_h(v_2)=w_h$, where $\Pi_h$ is
defined as in Lemma \ref{lem:Pih}. Then since $A_h(v_1-v_2)=0$, we
have by \eqref{eq:DPProperty1} that $a(u_h,v_1-v_2)=(f,v_1-v_2)$.
This implies $a(u_h,v_1)-(f,v_1)=a(u_h,v_2)-(f,v_2)$ and hence
$\sigma_h$ is well-defined.

\par
Choosing $w_h\geq 0$ in \eqref{eq:sigmahdef} and using
\eqref{eq:DPProperty} we conclude that $\sigma_h\leq 0$ on
$\bar\Omega$. Similarly \eqref{eq:sigmhaP2} follows from
\eqref{eq:DPProperty2}.
\end{proof}

\par
\noindent In view of the Lemma \ref{lem:sigmaprops} and since we
can chose $\Pi_h^{-1}(w_h)$ element-wise by
$\Pi_h^{-1}(w_h)|_T=w_h b_T/A_h(b_T)$, it is easy to see that we
can write \eqref{eq:sigmahdef} element-wise as

\begin{align}\label{eq:sigmahdef1}
\sigma_h|_T=\left(\int_T b_T\,dx\right)^{-1}\left(\int _T
fb_T\,dx-\int_T \nabla u_h\cdot \nabla b_T\,dx\right).
\end{align}

\par\noindent The above formula is useful in computing the
$\sigma_h$. Further, for any $v_h\in V_h$ we have by
\eqref{eq:sigmahdef} that
\begin{align*}
(\sigma_h,A_h(v_h))=(f,v_h)-a(u_h,v_h)\quad \forall\, v_h\in V_h.
\end{align*}
But since $(\sigma_h,v_h)=(\sigma_h, A_h(v_h))$, we finally have
\begin{align}\label{eq:sigmavh}
(\sigma_h,v_h)=(f,v_h)-a(u_h,v_h)\quad \forall\, v_h\in V_h.
\end{align}

\section{A Priori Error Analysis}\label{sec:Apriori}

The regularity theory of obstacle problem \cite[Theorem
2.5]{KS:2000:VI} implies that if $f\in L^2(\Omega)$, $\chi\in
H^{2}(\Omega)$ and $\Omega$ is convex, then the solution $u \in
H^{2}(\Omega)$. In particular the Lagrange multiplier $\sigma$
defined in \eqref{eq:sigmadef} can be written as $\sigma=f+\Delta
u$ and hence $\sigma\in L^{2}(\Omega)$.

\par
\noindent The following lemma follows from \eqref{eq:MP} and
\eqref{eq:sigmadef}, see \cite{KS:2000:VI, Glowinski:2008:VI}:
\begin{lemma}\label{lem:compatibility}
If $u\in H^2(\Omega)$, then $\sigma\in L^2(\Omega)$ and
\begin{align*}
\sigma &\leq 0 \quad \text{a.e. in}\quad \Omega,\\
(\sigma, u-\chi)&=0.
\end{align*}
Further if $u>\chi$ on some open set $D\subset \Omega$, then
$\sigma \equiv 0$ on $D$.
\end{lemma}

For the rest of this section, we assume that the data $f\in
H^1(\Omega)$, $\chi\in H^3(\Omega)$ and the solution $u \in
H^3(\Omega_N)\cup H^3(\Omega_C)$, where
\begin{align*}
\Omega_N&=\{x \in \Omega: u(x)>\chi(x) \},\\
\Omega_C&=\{x \in \Omega: u(x)=\chi(x) \}^\circ,
\end{align*}
for any set $D\subset \Omega$, the set $D^\circ$ denotes the
interior of $D$. Further assume that $u\in H^s(\Omega)$, where
$s=5/2-\epsilon$ for any $\epsilon >0$. We derive now an a priori
error estimate. This regularity assumption makes sense as the
solution of the obstacle problem looses the regularity at the free
boundary and if the free boundary is smooth the solution satisfies
as elliptic problem in the non contact region. Further, on the
contact region, the obstacle is assumed to be smooth enough.

\begin{theorem}\label{thm:Apriori}
There holds
\begin{align*}
\|\nabla(u-u_h)\|\leq C
h^{3/2-\epsilon}\left(\|u\|_{H^{5/2-\epsilon}(\Omega)}+\|f\|_{H^1(\Omega)}+\|\chi\|_{H^3(\Omega)}+\|u\|_{H^3(\Omega_N)}\right),
\end{align*}
for any $\epsilon >0$.
\end{theorem}

\begin{proof}
Since $I_hu\in\cK_h$, we find using \eqref{eq:DP} and integration by parts that
\begin{align*}
\|\nabla (u-u_h)\|^2&=a(u-u_h,u-I_hu)+a(u-u_h,I_hu-u_h)\\
&\leq a(u-u_h,u-I_hu)+a(u,I_hu-u_h)-(f,I_hu-u_h)\\
&= a(u-u_h,u-I_hu)+(-\Delta u-f,I_hu-u_h)\\
&= a(u-u_h,u-I_hu)-\sum_{T\in\cT_h}\int_T \sigma(I_hu-u_h)\,dx.
\end{align*}
The interpolation properties of $I_h$ in Lemma \ref{lem:Approx}
imply that
\begin{align}\label{eq:InterEst1}
\|\nabla(u-I_hu)\|\leq C h^{3/2-\epsilon}
|u|_{H^{5/2-\epsilon}(\Omega)},
\end{align}
for any $\epsilon >0$. On the other hand, divide the elements in
$\cT_h$ into the following sets:
\begin{align*}
\bN&=\{T\in \cT_h: u>\chi \text{ on } T\},\\
\bC&=\{T\in \cT_h: u\equiv \chi \text{ on } T\},\\
\bF&=\cT_h\backslash\{\bN\cup\bC\}.
\end{align*}
Then we write
\begin{align}
\sum_{T\in\cT_h}\int_T \sigma(I_hu-u_h)\,dx
&=\sum_{T\in\bN}\int_T
\sigma(I_hu-u_h)\,dx+\sum_{T\in\bC}\int_T
\sigma(I_hu-u_h)\,dx\notag\\
&\quad+\sum_{T\in\bF}\int_T \sigma(I_hu-u_h)\,dx\notag\\
&=\sum_{T\in\bC}\int_T \sigma(I_hu-u_h)\,dx+ +\sum_{T\in\bF}\int_T
\sigma(I_hu-u_h)\,dx,\label{eq:InterEst2}
\end{align}
since on any $T\in\bN$, we have $\sigma\equiv 0$ on $T$. Also
since $A_T(\sigma) \leq 0$ for any $T\in\cT_h$, we have
\begin{align*}
\int_T A_T(\sigma)(I_h\chi-u_h)\,dx \geq 0\quad\forall\,
T\in\cT_h.
\end{align*}
\par
\noindent Now let $T\in\bC$. Then we have $u\equiv \chi$ on $T$
and
\begin{align}
\int_T \sigma(I_hu-u_h)\,dx &=\int_T \sigma(I_h\chi-u_h)\,dx\geq \int_T (\sigma-A_T(\sigma))(I_h\chi-u_h)\,dx \notag\\
&= \int_T (\sigma-A_T(\sigma))(I_hu-u_h)\,dx \notag\\
&=\int_T (\sigma-A_T(\sigma))\big((I_hu-u_h)-A_T(I_hu-u_h)\big)\,dx\notag\\
&\geq -Ch_T^{2} \|\sigma\|_{H^1(T)} \|\nabla(I_hu-u_h)\|_{L^2(T)}.
\label{eq:InterEst3}
\end{align}
Finally let $T\in\bF$. Then using Lemma \ref{lem:compatibility},
we find
\begin{align*}
\int_T \sigma(I_hu-u_h)\,dx &=\int_T \sigma
\big((I_hu-u)+(u-\chi)+(\chi-I_h\chi)+(I_h\chi-u_h)\big)\,dx\\
&=\int_T \sigma \big(I_h(u-\chi)-(u-\chi)+(I_h\chi-u_h)\big)\,dx.
\end{align*}
Using the definition and interpolation properties of $I_h$, we
find
\begin{align}
\left|\int_T \sigma \big(I_h(u-\chi)-(u-\chi)\big)\,dx\right|&
=\left|\int_T
\big(\sigma-A_T(\sigma)\big)\big(I_h(u-\chi)-(u-\chi)\big)\,dx\right|\notag\\
&\leq C h_T^{1/2-\epsilon}\|(u-\chi)-I_h(u-\chi)\|_{L^2(T)}
\|\sigma\|_{H^{1/2-\epsilon}(T)}\notag\\
&\leq C
h_T^{3-2\epsilon}\|u-\chi\|_{H^{5/2-\epsilon}(T)}\,\|\sigma\|_{H^{1/2-\epsilon}(T)},\label{eq:FreeBterm}
\end{align}
for any $\epsilon>0$. As for $T\in \bC$, we note for any
$\epsilon>0$ that
\begin{align*}
\int_T \sigma(I_h\chi-u_h)\,dx &\geq \int_T (\sigma-A_T(\sigma))(I_h\chi-u_h)\,dx\\
&=\int_T (\sigma-A_T(\sigma))\big((I_h\chi-u_h)-A_T(I_h\chi-u_h)\big)\,dx\\
&\geq -Ch_T^{3/2-\epsilon} \|\sigma\|_{H^{1/2-\epsilon}(T)}
\|\nabla(I_h\chi-u_h)\|_{L^2(T)}.
\end{align*}
Using the triangle inequality and interpolation properties of
$I_h$, we find
\begin{align*}
\|\nabla(I_h\chi-u_h)\|_{L^2(T)}&\leq
\|\nabla(I_h\chi-\chi)\|_{L^2(T)}+\|\nabla(\chi-u)\|_{L^2(T)}+\|\nabla(u-u_h)\|_{L^2(T)}\\
&\leq C
h_T^{2}\|\chi\|_{H^3(T)}+\|\nabla(u-\chi)\|_{L^2(T)}+\|\nabla(u-u_h)\|_{L^2(T)}.
\end{align*}
Note that if $u-\chi=0$ on a set $D$ of measure non zero, then  by
the result of Stampachchia, $\nabla(u-\chi)=0$ a.e. on $D$, see
\cite[Appendix 4]{Kesavan}. Therefore
\begin{align*}
\|\nabla(u-\chi)\|_{L^2(T)}&=\left(\int_{\{u>\chi\}}
|\nabla(u-\chi)|^2\,dx\right)^{1/2}=\|\nabla(u-\chi)\|_{L^2(E)},
\end{align*}
where $E=\{x\in T: u(x)-\chi(x)>0\}$. Since $u-\chi\in
C(\bar\Omega)$, the set $E$ is open. From the assumption on the
regularity, we have $u-\chi\in H^3(E)$. Since $H^3(E)\subset
C^{1,\theta}(\bar E)$ with $\theta=1/2$, we have from
\cite[Theorem 2.4.5]{Kesavan} that
\begin{align*}
|\nabla(u-\chi)(x)|\leq C |x-x^*|^{1/2}\|u-\chi\|_{H^3(E)},
\end{align*}
where $x\in E$ and $x^*\in\partial E$ is such that
$\nabla(u-\chi)(x^*)=0$. Therefore
\begin{align*}
\|\nabla(u-\chi)\|_{L^2(T)}&\leq C |T|^{1/2}h_T^{1/2}
\|u-\chi\|_{H^3(E)}\leq C h_T^{2}\|u-\chi\|_{H^3(E)}.
\end{align*}
Therefore for any $T\in\bF$, we find
\begin{align}
\int_T \sigma(I_hu-u_h)\,dx &\geq -C
h_T^{3-2\epsilon}\|u-\chi\|_{H^{5/2-\epsilon}(T)}\,\|\sigma\|_{H^{1/2-\epsilon}(T)}\notag\\
&\quad -Ch_T^{3/2-\epsilon}
\|\sigma\|_{H^{1/2-\epsilon}(T)}\left(h_T^{2}\|\chi\|_{H^3(T)}+
h_T^{2}\|u-\chi\|_{H^3(E)}\right)\notag\\
&\quad-Ch_T^{3/2-\epsilon} \|\sigma\|_{H^{1/2-\epsilon}(T)}
\|\nabla(u-u_h)\|_{L^2(T)}, \label{eq:InterEst4}
\end{align}
where $E=\{x\in T: u(x)-\chi(x)>0\}$. We  complete the proof by
combining the estimates in
\eqref{eq:InterEst1}-\eqref{eq:InterEst4}
\end{proof}

\par
\noindent {\bf A priori error estimates for $\sigma_h$}.
In this section, we show that the discrete function $\sigma_h$
converges to $\sigma$ in the $H^{-1}$ norm at the same order of
convergence as that of the error $u-u_h$ in the $H^1$ norm. This is essential as the local efficiency estimates are derived using the combined norm of the error $u-u_h$ and
the dual norm of $\sigma-\sigma_h$.

\par
\noindent Let $(\cdot,\cdot)_T$ denote the $L^2(T)$-inner product.
Then from \eqref{eq:sigmadef} and \eqref{eq:sigmahdef}, we note
that
\begin{align}\label{eq:sigmaerror}
(\sigma-\sigma_h,b_T)_T=(\nabla (u_h-u),\, \nabla b_T)_T.
\end{align}

We prove the estimate in $H^{-1}$ norm. Let $D\subset\Omega$ be an
open set and for $v\in H^{-1}(D)$ define its $H^{-1}(D)$ norm
by
\begin{align*}
\|v\|_{H^{-1}(D)} =\sup_{\phi\in H^1_0(D),\;\phi\neq
0}\frac{\langle v ,\phi\rangle }{\|\nabla\phi\|_{L^2(D)}}.
\end{align*}

\par
\noindent
\begin{theorem}\label{thm:sigma-H-apriori}
Let $\sigma$ and $\sigma_h$ be defined by $\eqref{eq:sigmadef}$
and $\eqref{eq:sigmahdef}$. Then, there holds
\begin{align*}
\|\sigma-\sigma_h\|_{H^{-1}(T)}\leq C \left(
h_T\|\sigma-A_T(\sigma)\|_{L^2(T)}+\|\nabla(u-u_h)\|_{L^2(T)}\right).
\end{align*}
\end{theorem}
\begin{proof}
Using the triangle inequality, we write
\begin{align*}
\|\sigma-\sigma_h\|_{H^{-1}(T)}
&=\|\sigma-A_T(\sigma)\|_{H^{-1}(T)}+\|A_T(\sigma)-\sigma_h\|_{H^{-1}(T)}\\
&\leq
C\left(h_T\|\sigma-A_T(\sigma)\|_{L^2(T)}+\|A_T(\sigma)-\sigma_h\|_{H^{-1}(T)}\right).
\end{align*}
\par
\noindent  Let $\phi\in H^1_0(T)$ and $\phi_T=(1,\phi)_T$. Then
\begin{align*}
(A_T(\sigma)-\sigma_h,\phi)_T=(A_T(\sigma)-\sigma_h,\phi)_T=\phi_T
(1,b_T)_T^{-1} (A_T(\sigma)-\sigma_h,b_T)_T.
\end{align*}
Note that by scaling $|\phi_T| \leq C h_T^{3/2}\|
\phi\|_{L^2(T)}\leq C h_T^{5/2}\|\nabla \phi\|_{L^2(T)}$ and
$(1,b_T)_T^{-1}\leq C h_T^{-3}$. Therefore
$$|\phi_T(1,b_T)_T^{-1}|\leq C h_T^{-1/2}\|\nabla \phi\|_{L^2(T)}.$$
\noindent Further $\|b_T\|_{L^2(T)}\leq C h_T^{3/2}$ and $\|\nabla
b_T\|_{L^2(T)}\leq C h_T^{1/2}$. Using this we find
\begin{align*}
(A_T(\sigma)-\sigma_h,b_T)_T &=(A_T(\sigma)-\sigma,b_T)_T+(\sigma-\sigma_h,b_T)_T\\
&=(A_T(\sigma)-\sigma,b_T)_T+a(u_h-u,b_T)_T\\
&\leq C h_T^{3/2}\|A_T(\sigma)-\sigma\|_{L^2(T)} +C
h_T^{1/2}\|\nabla (u-u_h)\|_{L^2(T)},
\end{align*}
and
\begin{align*}
|\phi_T (1,b_T)_T^{-1} (A_T(\sigma)-\sigma_h,b_T)_T|\leq C \left(
h_T\|A_T(\sigma)-\sigma\|_{L^2(T)} +\|\nabla
(u-u_h)\|_{L^2(T)}\right)\|\nabla \phi\|_{L^2(T)},
\end{align*}
which proves
\begin{align*} \|A_T(\sigma)-\sigma_h\|_{H^{-1}(T)}\leq C \left(
h_T\|A_T(\sigma)-\sigma\|_{L^2(T)} +\|\nabla
(u-u_h)\|_{L^2(T)}\right).
\end{align*}
This completes the proof.
\end{proof}

\par
\noindent
\begin{theorem}\label{thm:sigma-L2-apriori}
Let $\sigma$ and $\sigma_h$ be defined by $\eqref{eq:sigmadef}$
and $\eqref{eq:sigmahdef}$. Then, there holds
\begin{align*}
\|\sigma-\sigma_h\|_{L^2(T)}\leq C \left(
\|A_T(\sigma)-\sigma\|_{L^2(T)}+h_T^{-1}\|\nabla(u-u_h)\|_{L^2(T)}\right).
\end{align*}
\end{theorem}
\begin{proof}
Using the triangle inequality, we find
\begin{align*}
\|\sigma-\sigma_h\|_{L^2(T)}
&=\|\sigma-A_T(\sigma)\|_{L^2(T)}+\|A_T(\sigma)-\sigma_h\|_{L^2(T)}.
\end{align*}
Since $w_T=A_T(\sigma)-\sigma_h|_T\in \bbP_0(T)$, $|(1,b_T)_T|
\approx C h_T^{3}$ and by the scaling arguments, we find for some
positive constant $C$ that
\begin{align*}
C\|A_T(\sigma)-\sigma_h\|_{L^2(T)}^2 &\leq (A_T(\sigma)-\sigma_h,
w_Tb_T)=(A_T(\sigma)-\sigma, w_Tb_T)+(\sigma-\sigma_h,
w_Tb_T)\\
&\leq \|\sigma-A_T(\sigma)\|_{L^2(T)}\|w_T\|_{L^2(T)}+|w_T|\,
a(u_h-u,b_T)\\
&\leq \|\sigma-A_T(\sigma)\|_{L^2(T)}\|w_T\|_{L^2(T)}+C h_T^{-1}
\|\nabla(u-u_h)\|_{L^2(T)}\|w_T\|_{L^2(T)}.
\end{align*}
This completes the proof.
\end{proof}

\section{A Posteriori Error Estimates}\label{sec:Aposteriori}
In this section, we derive residual based a posteriori error
estimates. Note that we assume $f\in L^2(\Omega)$ and $\chi\in H^1(\Omega)\cap C(\bar\Omega)$ and $\chi|_{\partial\Omega}\leq 0$ as in the introduction.  We begin by defining the following sets:
\begin{align*}
\bC_h=\{ T\in\cT_h: A_T(u_h)=A_T(\chi)\},
\end{align*}
and
\begin{align*}
\bN_h=\{ T\in\cT_h: A_T(u_h)>A_T(\chi) \}.
\end{align*}

\par
\noindent The residual based error estimates can be derived
conveniently by using the corresponding residual. Define the
residual $\cR_h: H^1_0(\Omega)\rightarrow \R$ by
\begin{align}\label{eq:Gh}
\cR_h(v)=a(u-u_h,v)+\langle \sigma-\sigma_h, v\rangle \quad \forall
v\in H^1_0(\Omega).
\end{align}

\par
\noindent The following lemma connects the norm of the residual
and the norms of the errors:

\begin{lemma}\label{lem:ErrorRelation}
There holds
\begin{align*}
\|\nabla(u-u_h)\|^2+\|\sigma-\sigma_h\|_{H^{-1}(\Omega)}^2&\leq
5\|\cR_h\|_{H^{-1}(\Omega)}^2-6\langle \sigma-\sigma_h,
u-u_h\rangle.
\end{align*}
\end{lemma}
\begin{proof}
Using \eqref{eq:Gh} and Young's inequality, we find
\begin{align*}
\|\nabla(u-  u_h)\|^2 &= a(u-  u_h, u-
u_h)=\cR_h(u-  u_h)-\langle \sigma-\sigma_h, u-  u_h\rangle\\
& \leq \|\cR_h\|_{H^{-1}(\Omega)}\;\|\nabla(u- u_h)\|-\langle
\sigma-\sigma_h,u-  u_h\rangle\\
&\leq
\frac{1}{2}\|\cR_h\|_{H^{-1}(\Omega)}^2+\frac{1}{2}\|\nabla(u-
u_h)\|^2-\langle \sigma-\sigma_h,u-  u_h\rangle,
\end{align*}
and
\begin{align}\label{eq:resi1-error1}
\|\nabla(u- u_h)\|^2 &\leq \|\cR_h\|_{H^{-1}(\Omega)}^2-2\langle
\sigma-\sigma_h,u- u_h\rangle.
\end{align}
Using again \eqref{eq:Gh}, we note that
\begin{align*}
\|\sigma-\sigma_h\|_{H^{-1}(\Omega)}\leq
\|\cR_h\|_{H^{-1}(\Omega)}+ \|\nabla(u-  u_h)\|.
\end{align*}
Now Young's inequality and \eqref{eq:resi1-error1} imply
\begin{align}
\|\sigma-\sigma_h\|_{H^{-1}(\Omega)}^2 &\leq
2\|\cR_h\|_{H^{-1}(\Omega)}^2 + 2
\|\nabla(u- u_h)\|^2 \notag\\
&\leq 4\|\cR_h\|_{H^{-1}(\Omega)}^2-4\langle \sigma-\sigma_h,u-
u_h\rangle. \label{eq:resi1-error2}
\end{align}
 The proof then follows by combining the estimates in \eqref{eq:resi1-error1}-\eqref{eq:resi1-error2}.
\end{proof}

\par
\smallskip
\noindent Define the following estimators:
\begin{align*}
\eta_1&=\left(\sum_{T\in \cT_h} h_T^2 \|\Delta
u_h+f-\sigma_h\|_{L^2(T)}^2\right)^{1/2},
\end{align*}
and
\begin{align*}
\eta_2&=\left(\sum_{e\in\cE_h^i}h_e\|\sjump{\nabla
u_h}\|_{L^2(e)}^2\right)^{1/2}.
\end{align*}

\par
\noindent
The norm of the residual $\cR_h$ is
estimated by using the error estimators in the following lemma:

\begin{lemma}\label{lem:GhBound}
It holds that
\begin{align*}
\|\cR_h\|_{H^{-1}(\Omega)}\leq C
\left(\eta_1^2+\eta_2^2\right)^{1/2}.
\end{align*}
\end{lemma}

\begin{proof}
Let $v\in H^1_0(\Omega)$ and choose $v_h\in V_h$ be such that
there holds the following approximation properties:
\begin{align*}
h_T^{-1}\|v-v_h\|_{L^2(T)}+\|\nabla v_h\|_{L^2(T)}\leq C \|\nabla
v\|_{L^2(\cT_T)},
\end{align*}
where $\cT_T$ is the union of triangles contained in patches of all three vertices of $T$. For example, $v_h$ can taken
to be a Scott-Zhang interpolation \cite{Scott:1990:Approx}. Then
\begin{align}
\langle \cR_h,v\rangle&=\langle \cR_h,v-v_h\rangle+\langle
\cR_h,v_h\rangle.
\end{align}
Firstly using \eqref{eq:Gh}, \eqref{eq:sigmadef} and
\eqref{eq:sigmavh}, we find
\begin{align*}
\langle \cR_h, v_h\rangle&=a(u-u_h, v_h)+\langle \sigma-\sigma_h, v_h \rangle\\
&=(f, v_h)-a(u_h, v_h)-(\sigma_h, v_h)=0.
\end{align*}
Secondly using \eqref{eq:Gh} and integration by parts,  we find
\begin{align*}
\langle \cR_h,v-v_h\rangle&=a(u-u_h,v-v_h)+\langle \sigma-\sigma_h,v-v_h \rangle\\
&=(f,v-v_h)-a(u_h,v-v_h)-(\sigma_h,v-v_h)\\
&=\sum_{T\in\cT_h} \int_{T}(f+\Delta u_h-\sigma_h)(v-v_h)\,dx-\sum_{T\in\cT_h} \int_{\partial T}\frac{\partial u_h|_T}{\partial n_T}(v-v_h)\,ds\\
&=\sum_{T\in\cT_h} \int_{T}(f+\Delta u_h-\sigma_h)(v-v_h)\,dx-\sum_{e\in\cE_h^i}\int_e\sjump{\nabla u_h}(v-v_h)\,ds\\
&\leq C \left(\eta_1^2+\eta_2^2\right)^{1/2}\|\nabla v\|.
\end{align*}
This completes the proof.
\end{proof}
It remains to find a lower bound for $\langle \sigma-\sigma_h,
u-u_h\rangle$. To this end,  let  $v^+=\max\{v,0\}$ and
$v^-=\max\{-v,0\}$ for any $v\in H^1(\Omega)$. Then $v=v^+ - v^-$.

\begin{lemma}\label{lem:sigmaLower}
There holds
\begin{align*}
\langle \sigma-\sigma_h, u-u_h\rangle &\geq -\frac{1}{12}
\|\sigma-\sigma_h\|_{H^{-1}(\Omega)}^2-C\left(\|\nabla(\chi-u_h)^+\|^2\right)+\sum_{T\in\bC_h}\int_T
\sigma_h(\chi-u_h)^-\,dx.
\end{align*}
\end{lemma}
\begin{proof}
Let $u_h^*=\max\{u_h,\chi\}$. Then $u_h^*\in\cK$ and
$u_h^*-u_h=(\chi-u_h)^+$. Using \eqref{eq:sigma} and $ab\leq
3a^2+b^2/12$, we find
\begin{align*}
\langle \sigma, u-u_h\rangle &= \langle \sigma,
u-u_h^*\rangle+\langle \sigma, u_h^*-u_h\rangle \geq \langle
\sigma, u_h^*-u_h\rangle\\
& = \langle \sigma-\sigma_h, u_h^*-u_h\rangle + \langle \sigma_h,
u_h^*-u_h\rangle\\
&\geq -\frac{1}{12}
\|\sigma-\sigma_h\|_{H^{-1}(\Omega)}^2-3\|\nabla(u_h^*-u_h)\|^2 +
\langle \sigma_h, u_h^*-u_h\rangle.
\end{align*}
\par
\noindent Therefore
\begin{align*}
\langle \sigma-\sigma_h, u-u_h\rangle &\geq -\frac{1}{12}
\|\sigma-\sigma_h\|_{H^{-1}(\Omega)}^2-3\|\nabla(\chi-u_h)^+\|^2 +
\langle \sigma_h, (u_h^*-u_h)-(u-u_h)\rangle.
\end{align*}
Using the fact that $\chi-u\leq 0$ a.e. in $\Omega$ and $
\sigma_h\leq 0$ on $\bar\Omega$, we find
\begin{align*}
\langle  \sigma_h, (u_h^*-u_h)-(u-u_h)\rangle \geq \langle
 \sigma_h, (u_h^*-u_h)-(\chi-u_h)\rangle.
\end{align*}
Note that as $(u_h^*-u_h)-(\chi-u_h)=(\chi-u_h)^-$, we have
\begin{align*}
\langle  \sigma_h, (u_h^*-u_h)-(u-u_h)\rangle \geq \langle
\ \sigma_h, (\chi-u_h)^-\rangle=\int_\Omega
\sigma_h(\chi-u_h)^-\,dx.
\end{align*}
Now using Lemma \ref{lem:sigmaprops},
\begin{align*}
\int_\Omega  \sigma_h(\chi-u_h)^-\,dx=
\sum_{T\in\bC_h}\int_T \sigma_h(\chi-u_h)^-\,dx.
\end{align*}
This completes the proof.
\end{proof}
From Lemma \ref{lem:ErrorRelation}, Lemma \ref{lem:GhBound} and
Lemma \ref{lem:sigmaLower}, we deduce the following result on {\em
a posteriori} error control of quadratic FEM:

\begin{theorem}\label{thm:Apost-Est} It holds that
\begin{align*}
\|\nabla(u-u_h)\|^2+\|\sigma-\sigma_h\|_{H^{-1}(\Omega)}^2&\leq
C\Big(\eta_1^2+\eta_2^2+\|\nabla(\chi-u_h)^+\|^2
-\sum_{T\in\bC_h}\int_T \sigma_h(\chi-u_h)^-\,dx \Big).
\end{align*}
\end{theorem}

The following local efficiency estimates can be proved easily
using the bubble function techniques and the definition of
$\sigma$ in \eqref{eq:sigmadef}:
\begin{lemma}\label{lem:Efficiency}
There hold
\begin{align*}
&h_T\|f+\Delta u_h-\sigma_h\|_{L^2(T)}\leq C \left(\|\nabla
(u-u_h)\|_{L^2(T)}+\|\sigma-\sigma_h\|_{H^{-1}(T)}+h_T\inf_{\bar f\in\bbP_0(T)}\|f-\bar f\|_{L^2(T)}\right)\\
&h_e^{1/2}\|\sjump{\nabla u_h}\|_{L^2(e)}\leq C \left(\|\nabla
(u-u_h)\|_{L^2(T_e)}+\|\sigma-\sigma_h\|_{H^{-1}(T_e)}+h_e\inf_{\bar
f\in\bbP_0(T_e)}\|f-\bar f\|_{L^2(T_e)}\right),
\end{align*}
where $T_e$ is the patch of the face $e\in \cE_h^i$.
\end{lemma}

\par
\noindent
\begin{remark}
In view of Braess \cite{Braess:2005:VI}, the terms $$\|\nabla(\chi-u_h)^+\|\quad\text{and}\quad-\sum_{T\in\bC_h}\int_T \sigma_h(\chi-u_h)^-\,dx$$
are of higher order.
\end{remark}

\section{Numerical Implementation}\label{sec:Numerical}
In this section, first we discuss the primal-dual active set
method and then present numerical experiment.

\subsection{Implementation procedure}
We propose the primal-dual active set method for the numerical
experiments. In the light of the algorithm in
\cite{HK:2003:activeset}, we develop the following algorithm for
solving the 3D-obstacle problem by the quadratic finite element
method developed in this article. For the given mesh size $h$, let
$\cT_h$ be the simplicial triangulation of $\Omega\subset \R^3$
with number of simplices denoted by $M$. Let the simplices be
enumerated by $\{T_j\}_{\{1\leq j\leq M\}}$. Let $N$ be the
dimension of $V_h$ and $\{\phi_i\}_{\{1 \leq i\leq N\}}$ be its
basis. Denote by $A=[A_{ij}]_{\{1\leq i,j\leq N\}}$ the stiffness
matrix, where
\begin{align*}
A_{ij}=(\nabla\phi_j, \nabla\phi_i).
\end{align*}
Define the matrix $B=[B_{ij}]_{\{1\leq i\leq M,1\leq j\leq N\}}$
where
\begin{align*}
B_{ij}=\frac{1}{|T_j|}\int_{T_j}\phi_i\,dx.
\end{align*}
The load vector $b=[b_i]_{\{1\leq i\leq N\}}$ is defined as
\begin{align*}
b_{i}=(f,\phi_i).
\end{align*}
Also define $\gamma=[\gamma_j]_{\{1\leq j\leq M\}}$, where
\begin{align*}
\gamma_j=\frac{1}{|T_j|}\int_{T_j}\chi\,dx.
\end{align*}

\par
\noindent Let the discrete solution $u_h\in V_h$ be represented by
\begin{align*}
u_h=\sum_{i=1}^N\alpha_i\phi_i.
\end{align*}

\par
\noindent The Lagrange multiplier $\sigma_h\in V_{pc,h}$ which
will be written as
\begin{align*}
\sigma_h=\sum_{j=1}^M \beta_j \psi_j,
\end{align*}
where $\psi_j$ is the characteristic function of $T_j$ defined by
\begin{align*}
A\alpha+B\beta=b,
\end{align*}
with $\alpha=[\alpha_i]_{\{1\leq i\leq N\}}$ and $\beta=[\beta_j]_{\{1\leq j\leq M\}}$.
The complementarity conditions are given by
\begin{align}
\beta^T (B^T\alpha-\gamma)=0,\quad \beta\leq 0,\quad \text{and }B^T\alpha-\gamma\geq 0.
\end{align}

\par
\noindent
The complementarity conditions can be written as
\begin{align*}
C(\alpha,\beta)=0,
\end{align*}
where
\begin{align*}
C(\alpha,\beta)=\beta-\min\{0,\beta+c(B^T\alpha-\gamma)\},
\end{align*}
for some $c>0$. Finally let $\Lambda=\{1,2,\cdots, M\}$ be the index set
of mesh elements $T_j\in\cT_h$.

\smallskip
\par
\noindent The primal-dual active set algorithm is defined as
follows:

\begin{algorithm}\label{alg:1}
Initialize $\alpha^0$ and $\beta^0$. Let $k=1$,
$\alpha^1=\alpha^0$ and $\beta^1=\beta^0$. For $k\geq 1$, perform
the following  steps:

\item[{\em Step 1.}] Find $A_k=\{j\in\Lambda :
\beta_j^k+c(B^T\alpha^k-\gamma)_j<0\}$ and $I_k=\{j\in \Lambda:
\beta_j^k+c(B^T\alpha^k-\gamma)_j>=0\}$.

\item[{\em Step 2.}] Solve the system
\begin{align*}
A\alpha+B\beta &=b \\
B^T\alpha&=\gamma \quad\text{ on } A_k\\
\beta &=0\quad\text{ on } I_k.
\end{align*}
\item[{\em Step 3.}] Stop or set $k=k+1$, $\alpha^k=\alpha$ and
$\beta^k=\beta$, where $\alpha$ and $\beta$ are the solutions of
the system in {\em Step 2}.
\end{algorithm}

\subsection{Numerical Experiments}\label{sec:Numerics}
In this section, we present some numerical experiments to
illustrate the theoretical results derived in this article. For
this, we consider the computational domain to be the unit cube
$\Omega=(0,1)^3$ in $\R^3$ and the obstacle function to be
$\chi\equiv 0$. Further the force function $f$ is taken as
\begin{equation}\label{eq:uht}
f(x,y,z) := \left\{ \begin{array}{ll} -4(2r^2+3(r^2-r_0^2)) & \text{ if } r>r_0,\\\\
-8r_0^2(1-r^2+r_0^2) & \text{ if } r\leq r_0,
\end{array}\right.
\end{equation}
where $r=(x^2+y^2+z^2)^{1/2}$ and $r_0=7$. The nonhomogeneous
Dirichlet boundary condition is taken in such a way that the
solution $u$ is given by $u(x,y,z)=(\max(r^2-r_0^2,0))^2$. The
{\em Algorithm} \ref{alg:1} is used in computations with $c=1$ in
the {\em Step 1} therein. In the experiment, we compute the order
of convergence in the energy norm to test the performance of the
result in Theorem \ref{thm:Apriori}. We begin with an initial mesh
given in Figure \ref{fig:imesh} and generate an array of uniformly
refined meshes by tetrahedrons by dividing each tetrahedron into
12 tetrahedrons. We compute the discrete solution on these meshes
and then compute the corresponding errors using a quadrature formula
that is exact for cubic polynomials. The results are depicted in
the Table \ref{table:H1e}. The results match closely with the
theoretical results. We have developed our in-house MATLAB code
for this experiment. The discrete  and the exact (interpolation)
solutions are plotted in Figure \ref{fig:solu} on the mesh with
mesh $h=0.433$ (around 1.03 Lakh tetrahedrons).

\begin{figure}[!ht]
\begin{center}
\includegraphics[width=8cm,height=7cm]{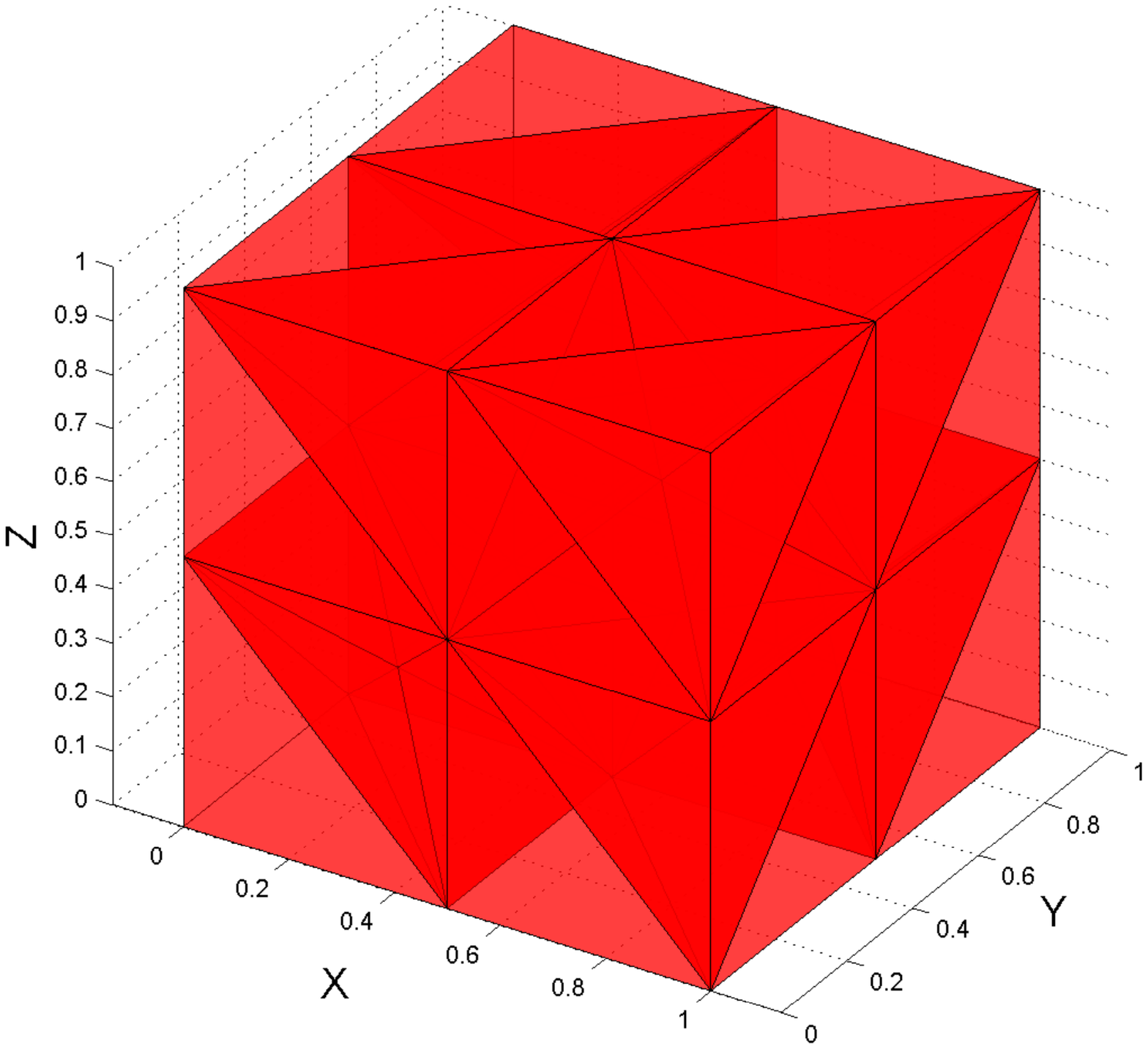}
\caption{The initial mesh in computations}\label{fig:imesh}
\end{center}
\end{figure}

\begin{table}[h!!]
 {\small{\footnotesize
\begin{center}
\begin{tabular}{|c|c|c|}\hline
 $h$ & $\|\nabla(u-u_h)\|_{L^2(\Omega)} $  & order  \\
\hline\\[-12pt]
 0.3467   & 1.8500e-001  &  --       \\
 0.1733   & 5.6046e-002  & 1.3596    \\
 0.0867   & 1.9210e-002  & 1.4112     \\
 0.0433   & 7.1151e-003  & 1.3636    \\
\hline
\end{tabular}
\end{center}
}}
\par\medskip
\caption{Errors and order of convergence in  $H^1$ norm }
\label{table:H1e}
\end{table}

\begin{figure}[!ht]
\begin{center}
\includegraphics[width=10cm,height=7cm]{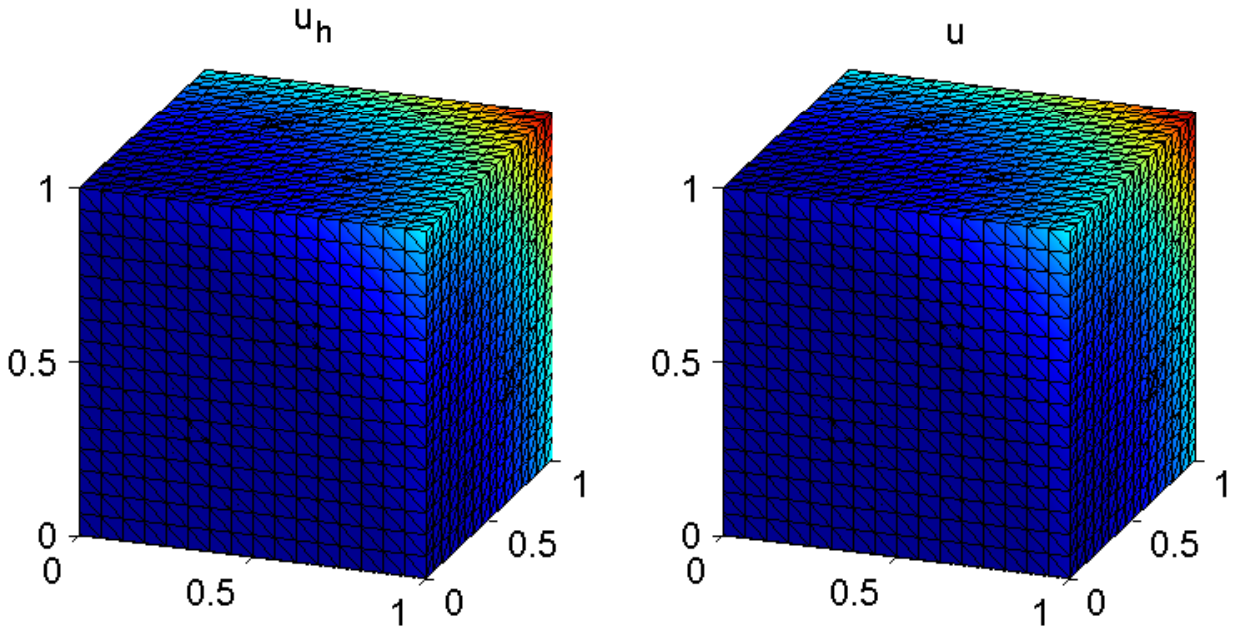}
\caption{The comparison between the computed (left) and the exact
(right) solution}\label{fig:solu}
\end{center}
\end{figure}

\par
\noindent
 Numerical experiments to test the performance
of a posteriori error estimates will be discussed in the future
work.

\section{Conclusions}\label{sec:Conclusions}
We have developed a quadratic finite element method for the three
dimensional elliptic obstacle problem. The finite element space is
constructed by using the standard $P_2$ Lagrange finite element
and a space of element-wise bubble functions. This enables us to
prove optimal order (with respect to the regularity) error
estimate in the energy norm. A posteriori error estimates are
derived by constructing a suitable Lagrange multiplier. Further, a
primal-dual active set method is proposed for the numerical
implementation and a numerical experiment is presented to
illustrate the theoretical result on {\em a priori} error
estimate.

%

\end{document}